\documentclass[11pt,a4]{article}
\usepackage{latexsym}
\usepackage{amsmath}
\usepackage{amssymb}
\usepackage{amsthm}
\usepackage{amscd}
\usepackage{color}
\usepackage{citeref}
\usepackage{comment}
\definecolor{red}{rgb}{0.7,0,0}

\numberwithin{equation}{section}

\newtheorem{theorem}{Theorem}[section]
\newtheorem{lemma}[theorem]{Lemma}

\newtheorem{proposition}[theorem]{Proposition}

\theoremstyle{definition}
\newtheorem{definition}[theorem]{Definition}

\theoremstyle{remark}

\allowdisplaybreaks

\begin{document}

\title{Compactness of multilinear commutators generated by VMO functions and fractional integral operator on Morrey spaces}

\author{Daiki Takesako\footnote{Address: Department of Mathematics Graduate School of Science and Engineering, Chuo University, 1-13-27, Bunkyo-ku, 112-8551, Tokyo, Japan. e-mail: takesako.math@gmail.com}}


\date{}

\maketitle

\begin{abstract}
The aim of this paper is to improve the compactness of the multilinear commutators in Morrey spaces generated by VMO functions and fractional integral operators. In this paper, we will use the decomposition of the tilde closed subspaces of Morrey spaces. This gives us more understanding about commutators.
\end{abstract}

{\bf keywords} Morrey space, compact operator, fractional integral operator, multilinear commutator, tilde closed subspace.

{\bf 2020 Classification} 42B35, 42B20, 46B50

\section{Introduction}
The aim of this paper is to refine the known results of the multilinear commutators in Morrey spaces generated by VMO functions and the fractional integral operator $I_{\alpha}$. First, we recall the definition of Morrey space ${\mathcal M}^{p}_q({\mathbb R}^n)$ and the fractional integral operator $I_{\alpha}$.
\begin{definition}
	We denote the open ball whose center is $x$ with radius $r>0$ by $B(x,r)$. Let $1\le q \le p < \infty$. For an $L^{q}_{\rm loc}({\mathbb R}^n)$-function $f$, its Morrey norm $\| \cdot \|_{{\mathcal M}^{p}_q}$ is defined by
	\begin{equation*}
		\| f \|_{{\mathcal M}^{p}_q({\mathbb R}^n)}
		:=
		\sup\limits_{(x,r) \in {\mathbb R}^{n} \times \mathbb{R}_+}
		|B(x,r)|^{\frac{1}{p}-\frac{1}{q}}
		\left(
		\int_{B(x,r)}|f(y)|^{q}{\rm d}y
		\right)^{\frac{1}{q}}.
	\end{equation*}
	The Morrey space ${\mathcal M}^{p}_q({\mathbb R}^n)$ is the set of $L^{q}_{\rm loc}({\mathbb R}^n)$-functions with this norm finite.
\end{definition}

Note that, for $p=q$, we have $\mathcal{M}^p_p(\mathbb{R}^n)$ is the Lebesgue space $L^p(\mathbb{R}^n)$. By H\"{o}lder's inequality, we get for $1 \leq q_1 \leq q_2 \leq p < \infty$, $\| f \|_{{\mathcal M}^{p}_{q_1}} \leq \| f \|_{{\mathcal M}^{p}_{q_2}} \leq \| f \|_{{\mathcal M}^{p}_p} = \| f \|_{L^p}$.\\

\begin{definition}
	For a measurable function $f$ and $0 < \alpha < n$, we write $I_{\alpha} $ as the fractional integral operator of order $\alpha$ defined by
	\begin{align*}
		I_{\alpha}f(x):=\int_{{\mathbb R}^n} \frac{f(y)}{|x-y|^{n-\alpha}}dy,
	\end{align*}
	as long as the right-hand side makes sense.
\end{definition}


We define the commutator generated by a multiplication operator by a function $a_1$ and the fractional integral operator as $[a_{\{1\}},I_{\alpha}]f:=a_1I_{\alpha}(f)-I_{\alpha}(a_1f)$. We will consider the multilinear commutators formed as $[a_{\{1,...,l\}},I_{\alpha }]$ that are nests of commutators: $[a_{\{1,...,l\}},I_{\alpha}]:=[a_1,[a_{\{2,...,l\}},I_{\alpha}]]$. We assume the   $a_j$'s are BMO-functions defined as follows:


\begin{definition}
We say an $L_{loc}^1({\mathbb R}^n)$-function $b$ is in ${\rm BMO}({\mathbb R}^n)$ when
	\[
		\|b\|_* 
		:=
		\sup\limits_{(x,r) \in {\mathbb R}^{n} \times \mathbb{R}_+}
		\frac{1}{|B(x,r)|}\int_{B(x,r)}|b(y)-m_{B(x,r)}(b)|{\rm d}y<\infty.
	\]
where $m_{B(x,r)}(b)$ denotes the average of $b$ over $B(x,r)$.
\end{definition}

As a consequence of the definitions, we can write $[{a_{\{1,...,l\}}},I_{\alpha}]$ as follows:

\begin{definition}
	Let $l \in \mathbb N$, $a_j \in {\rm BMO}({\mathbb R}^n)$, $j=1,...,l$, $0<\alpha<n$, $1 < q \le p < \infty$ and $f \in {\mathcal M}^{p}_q({\mathbb R}^n)$. Define the multilinear commutator $[a_{\{1, ... ,l\}},I_{\alpha}]$ by 
	\[
		[a_{\{1, ... ,l\}},I_{\alpha}]f(x):=\int_{{\mathbb R}^n}\prod_{j=1}^l(a_j(x)-a_j(y))\frac{f(y)}{|x-y|^{n-\alpha}}dy.
	\]
\end{definition}

Additional conditions of the main theorem are $a_1 \in {\rm VMO}({\mathbb R}^n)$, and the target space is the {\it tilde closed subspace} $\widetilde{\mathcal M}^{s}_t({\mathbb R}^n)$ defined below.

\begin{definition}
	We set $C_c^{\infty}({\mathbb R}^n)$ as the set of $C^{\infty}({\mathbb R}^n)$-functions with compact support. The space ${\rm VMO}({\mathbb R}^n)$ is the closure of $C_c^{\infty}({\mathbb R}^n)$ under the norm of ${\rm BMO}({\mathbb R}^n)$. 
	Also, The subspace $\widetilde{\mathcal M}^{p}_q({\mathbb R}^n)$ is the closure of  $C_c^{\infty}({\mathbb R}^n)$ under the norm of $ {\mathcal M}^{p}_q({\mathbb R}^n)$.
\end{definition}

We will establish the following theorem.
\begin{theorem}\label{main}
	Suppose $1 < q \le p < \infty$, $l \in \mathbb{N}$, and $a_1 \in {\rm VMO}({\mathbb R}^n)$. Take $s, t$ so that
		\[
			\frac{p}{q}=\frac{s}{t} \quad {\rm and} \quad \frac{1}{s}=\frac{1}{p}-\frac{\alpha}{n} > 0.
		\]
	Then the operator $[a_{\{1, ... ,l\}},I_{\alpha}]$ is compact from ${\mathcal M}^{p}_q({\mathbb R}^n)$ to 
	$\widetilde{\mathcal M}^{s}_t({\mathbb R}^n)$.
\end{theorem}
This theorem gives us a possibility of approximation of multilinear commutators. Thus we learn that these commutators differ from other compact ${\mathcal M}^{p}_q({\mathbb R}^n)$-${\mathcal M}^{s}_t({\mathbb R}^n)$ operators.

The study of compactness for commutators on Morrey and Morrey-type spaces has become an important theme in modern harmonic analysis. A significant early contribution was made by Sawano and Shirai~\cite{mlc-cpt}, who demonstrated that compact commutator estimates remain valid even in settings with non-doubling measures. Their work highlighted the robustness of Morrey-space techniques beyond classical Euclidean frameworks and provided a strong foundation for subsequent developments.\\
Motivated by these ideas, Chen, Ding and Wang~\cite{CDW09} established compactness criteria for commutators of Riesz potentials on classical Morrey spaces. They showed that boundedness of such commutators is governed by BMO-type regularity of the symbol, whereas compactness requires a finer VMO-type condition. This distinction has since been recognized as a central principle underlying compactness phenomena in Morrey-space analysis.\\
Generalized Morrey spaces, which allow spatially dependent growth functions, were studied in~\cite{CDW10}, where compactness of commutators of Littlewood--Paley operators were obtained. This work extended the classical theory to a broader and more flexible functional setting. Parallel progress occurred for analytic singular integrals: Tao, Yang and Yang~\cite{TYY19} established compactness results for commutators of the Cauchy integral, while Dao, Duong and Ha~\cite{DDH21} investigated Cauchy--Fantappi\`e type integrals on complex ellipsoids, demonstrating that Morrey-type compactness naturally extends to complex-analytic and geometric contexts.\\
Weighted Morrey spaces form another important direction of research. Liu and Li~\cite{JL} studied commutators of Calderón--Zygmund operators in the weighted setting and obtained corresponding compactness characterizations. Likewise, Gong, Vempati, Wu and Xie~\cite{GVW22} obtained boundedness and compactness results for Cauchy-type integrals on weighted Morrey spaces. These works clarified the role of weighted VMO-type conditions in controlling compactness under non-uniform density conditions.\\
A unifying and structural perspective was later provided by Bokayev, Burenkov, Matin and Adilkhanov~\cite{B24}, who developed general compactness principles for global Morrey-type spaces. Their approach not only yielded new results for Riesz potentials but also provided a framework integrating several earlier contributions into a single conceptual viewpoint.

\section{Preliminaries}
In this paper, we use the symbol ``$A \lesssim B$'' when there exists a constant $C > 0$ independent of $A,B$ which satisfies $A \le B$. We write $A \sim B$ when $A \lesssim B$ and $B \lesssim A$ hold true. If the constant $C$ depends on $D$, we write $A \lesssim_D B$. Also, remark that we do not write $p'$ as $\frac{1}{p}+\frac{1}{p'}=1$ unless otherwise stated.

\subsection{Characterization and the decomposition of closed subspace of Morrey spaces} 
In this subsection, we will see the definitions and properties of the closed subspaces of Morrey spaces.
We write $L^0_c({\mathbb{R}}^n)$ for the set of measurable functions with compact support. 

\begin{definition}{(closed subspaces of Morrey spaces)}
Let $1 \le q \le p < \infty$.
	\begin{enumerate}
			\item[$(1)$]
				The {\it tilde-closed subspace}
				$\widetilde{\mathcal M}^{p}_q({\mathbb R}^n)$
				is defined to be the closure
				of $C^\infty_{\rm c}({\mathbb R}^n)$
				under the norm in ${\mathcal M}^{p}_q({\mathbb R}^n)$.

			\item[$(2)$]
				The {\it bar subspace}
				$\overline{\mathcal M}{}^p_q({\mathbb R}^n)$
				is defined to be the closure of
				$L^{\infty}({\mathbb R}^n) \cap {\mathcal M}^p_q({\mathbb R}^n)$
				under the norm in
				${\mathcal M}^p_q({\mathbb R}^n)$.
				\index{$\overline{\mathcal M}{}^p_q({\mathbb R}^n)$}

			\item[$(3)$]
				The {\it star subspace} $\overset{*}{\mathcal M}{}^p_q({\mathbb R}^n)$  is defined to be the closure of $L^0_{\rm c}({\mathbb R}^n) \cap {\mathcal M}^p_q({\mathbb R}^n)$ under the norm in ${\mathcal M}^p_q({\mathbb R}^n)$.
		\end{enumerate}
\end{definition}

In \cite{MS1}, we have the characterization of the {\it tilde-closed subspace} $\widetilde{\mathcal M}^{p}_q({\mathbb R}^n)$, the {\it bar subspace} $\overline{\mathcal M}{}^p_q({\mathbb R}^n)$, and the {\it star subspace} $\overset{*}{\mathcal M}{}^p_q({\mathbb R}^n)$.

\begin{proposition}{\rm (\cite[Proposition 331,332,333]{MS1})}\label{char}
	Let $1\le q\le p<\infty$. Then,
		\[
			\widetilde{\mathcal M}^{p}_q({\mathbb R}^n)
			=\{f \in {\mathcal M}^p_q({\mathbb R}^n) : \lim\limits_{R\to \infty} \|f\chi_{\{|f|> R\}\cup(\mathbb{R}^n\setminus B(R))}\|_{\mathcal{M}^p_q({\mathbb R}^n)}=0 \},
		\]
		\[
			\overline{\mathcal M}{}^p_q({\mathbb R}^n)
			=\{f \in {\mathcal M}^p_q({\mathbb R}^n) : \lim\limits_{R\to \infty} 
			\|f\chi_{[R,\infty)}(|f|)\|_{\mathcal{M}^p_q({\mathbb R}^n)}=0\},
		\] 
		\[
			\overset{*}{\mathcal M}{}^p_q({\mathbb R}^n)
			=\{f \in {\mathcal M}^p_q({\mathbb R}^n) : \lim\limits_{R\to \infty} 
			\|f\chi_{\mathbb{R}^n\setminus B(R)}\|_{\mathcal{M}^p_q({\mathbb R}^n)}=0\}.
		\]
\end{proposition}
From the proposition above, we immediately get the decomposition below. (See \cite{MS1})
\begin{proposition}{\rm (\cite[Corollary 334]{MS1})}\label{decom}
	Let $1 \le q \le p < \infty$. Then,
	\begin{equation*}
		\widetilde{\mathcal M}^p_q({\mathbb R}^n)
		=
		\overline{\mathcal M}{}^p_q({\mathbb R}^n)
		\cap
		\overset{*}{\mathcal M}{}^p_q({\mathbb R}^n).
	\end{equation*}
\end{proposition}
Also, the bar subspace $\overline{\mathcal M}{}^p_q({\mathbb R}^n)$ enjoys a useful property \cite{2025};
\begin{proposition}\label{barsubsp}{\rm (\cite[Lemma 6]{2025})}
	If
	\begin{equation*}
		1<q \le p<\infty, \quad 1<t \le s<\infty, \quad p<s, \quad {\rm and} \quad \frac{q}{p}=\frac{t}{s}, \quad
	\end{equation*}
	then
	\[
		{\mathcal M}^p_q({\mathbb R}^n) \cap {\mathcal M}^s_t({\mathbb R}^n) \subset \overline{{\mathcal M}}{}^p_q({\mathbb R}^n).
	\]
\end{proposition}

\subsection{Properties of operators}
Here, we recall some previous studies. The first one is about the fractional integral operator $I_{\alpha}$.
It is known as Adams' theorem such that the fractional integral operator $I_{\alpha}$ is bounded from the Morrey space ${\mathcal M}^{p}_q({\mathbb R}^n)$ to the Morrey space ${\mathcal M}^{s}_t({\mathbb R}^n)$.
\begin{proposition}{\rm(\cite{Adams})}\label{Adams}
	Suppose $1 < q \le p < \infty$ and $0 < \alpha < n$. Then, for any $f \in {\mathcal M}^{p}_q({\mathbb R}^n)$, the integration in $I_{\alpha}$ converges absolutely almost all $x \in {\mathbb{R}}^n$ and satisfies the following estimate with a constant $C$ independent of $f$:
	\[
		\|I_{\alpha}f\|_{{\mathcal M}^{s}_t({\mathbb R}^n)} \le C \|f\|_{{\mathcal M}^{p}_q({\mathbb R}^n)},
	\]
where $1<t \le s < \infty$ satisfies the following:
	\[
		\frac{1}{s}=\frac{1}{p}-\frac{\alpha}{n} \quad {\rm and} \quad \frac{p}{q}=\frac{s}{t}.
	\]
\end{proposition}

The study of boundedness of $[a_{\{1, ... ,l\}},I_{\alpha}]$ dates back to 2005; see the work of Dongyong Yang and Yan Meng \cite{mlc-bdd}.

\begin{proposition}{\rm (\cite{mlc-bdd})}\label{mlc-bdd}
	Let $l \in \mathbb N$, $j=1,...,l$, $a_j \in {\rm BMO}({\mathbb R}^n)$, $1 < q \le p < \infty$, $0 < \alpha < n$, $\frac{1}{s}=\frac{1}{p}-\frac{\alpha}{n}$ and $\frac{p}{q}=\frac{s}{t}$. Then
	\begin{align*}
		\| [a_{\{1, ... ,l\}},I_{\alpha}] \|_{{\mathcal M}^{p}_q({\mathbb R}^n) \rightarrow {\mathcal M}^{s}_t({\mathbb R}^n)}
		\le C \prod_{j=1}^l \|a_j\|_*
	\end{align*}
	where the constant $C$ is independent of $\{a_j\}$, $j=1,...,l$.

\end{proposition}

In 2008, Sawano and Shirai showed the compactness property of $[a_{\{1, ... ,l\}},I_{\alpha}]$.

\begin{proposition}{\rm (\cite{mlc-cpt})}
	If $a_1 \in {\rm VMO}({\mathbb R}^n)$ and $a_j,l,p,q,s,t,\alpha$ as in the proposition above, then 
	\begin{align*}
		[a_{\{1, ... ,l\}},I_{\alpha}]:{\mathcal M}^{p}_q({\mathbb R}^n) \rightarrow {\mathcal M}^{s}_t({\mathbb R}^n)
	\end{align*}
	is a compact operator.
\end{proposition}

Another preliminary connected to the compactness of $[a_{\{1, ... ,l\}},I_{\alpha}]$ is the following theorem proven by Sawano, Hakim and the present author.

\begin{proposition}{\rm (\cite{2025})}\label{l1}
	Let $l=1$, $0<\alpha<n$, $1<q \le p<\infty$, $1<t \le s<\infty$, and $a_1 \in {\rm VMO}({\mathbb R}^n)$. Assume that
	\[
		\frac{1}{s}=\frac{1}{p}-\frac{\alpha}{n}  \quad {\rm and}\quad \frac{q}{p}=\frac{t}{s}.
	\]
	Then the commutator $[a_{\{1\}},I_\alpha]$ is compact from ${\mathcal M}^p_q({\mathbb R}^n)$ to $\widetilde{\mathcal M}^s_t({\mathbb R}^n)$.
\end{proposition}

\section{Lemmas for the proof}
In this paper, we use assumptions of Morrey norm about BMO-functions twice. Thus, we will check it in this subsection.
\begin{lemma}\label{BMO dyadic}
	If
	\begin{align*}
		1 \le q \le p < \infty, \quad g \in {\rm BMO}({\mathbb R}^n), \quad k \in \mathbb N \cup \{0\}, \quad \alpha \ge \beta > 0, \quad and \quad y \in B(0,\beta),
	\end{align*}
	then
	\begin{align*}
		&\| \chi_{B(0,2^{k+1}\alpha) \backslash B(0,2^k\alpha)} |g-g(y)| \|_{{\mathcal M}{}^p_q({\mathbb R}^n)}\\
		&\lesssim (2^{k}\alpha)^{\frac{n}{p}} \left\{ \|g\|_* \log_2\left(\frac{2^{k+1}\alpha}{\beta}\right) + |g(y)-m_{B(0,\beta)}(g)| \right\}.
	\end{align*}
\end{lemma}

\begin{proof}
Let $J \in \mathbb N$ is the smallest number which satisfies
	\[
		B(0,2^{k+1}\alpha) \backslash B(0,2^k\alpha) \subset 2^JB(0,\beta).
	\] 
Therefore, $J$ satisfies  
	\begin{align}
		2^J \beta \sim 2^k \alpha, \quad and \quad J \sim  \log_2\left(\frac{2^{k+1}\alpha}{\beta} \right). \label{123}
	\end{align}
By the triangle inequality, we have 
	\begin{align*}
		&\| \chi_{B(0,2^{k+1}\alpha) \backslash B(0,2^k\alpha)} |g-g(y)| \|_{{\mathcal M}{}^p_q({\mathbb R}^n)}\\
		&\le \| \chi_{B(0,2^{k+1}\alpha) \backslash B(0,2^k\alpha)} |g-g(y)| \|_{L^p({\mathbb R}^n)}\\
		&\le \| \chi_{B(0,2^{k+1}\alpha) \backslash B(0,2^k\alpha)} |g(y)-m_{B(0,\beta)}(g)| \|_{L^p({\mathbb R}^n)}\\
			&\quad+ \| \chi_{B(0,2^{k+1}\alpha) \backslash B(0,2^k\alpha)} |g-m_{B(0,\beta)}(g)| \|_{L^p({\mathbb R}^n)}\\
		&\le \| \chi_{B(0,2^{k+1}\alpha) \backslash B(0,2^k\alpha)} |g(y)-m_{B(0,\beta)}(g)| \|_{L^p({\mathbb R}^n)}\\
			& \quad + \| \chi_{B(0,2^{k+1}\alpha) \backslash B(0,2^k\alpha)} |g-m_{2^JB(0,\beta)}(g)| \|_{L^p({\mathbb R}^n)} \\
			& \quad + \sum_{m=1}^{J} \| \chi_{B(0,2^{k+1}\alpha) \backslash B(0,2^k\alpha)} |m_{2^mB(0,\beta)}(g)-m_{2^{m-1}B(0,\beta)}(g)| \|_{L^p({\mathbb R}^n)}.
	\end{align*}
By the inequality $\|a_j\|_* \sim \sup\limits_{(x,r) \in {\mathbb R}^{n} \times \mathbb{R}_+} |B(x,r)|^{-\frac{1}{p} } \left\| \chi_{B(x,r)}|a_j-m_{B(x,r)}(a_j)| \right\|_{L^p({\mathbb R}^n)}$, we get
	\begin{align*}
		&\| \chi_{B(0,2^{k+1}\alpha) \backslash B(0,2^k\alpha)} |g-g(y)| \|_{{\mathcal M}{}^p_q({\mathbb R}^n)}\\
		&\lesssim (2^{k+1}\alpha)^{\frac{n}{p}} |g(y)-m_{B(0,\beta)}(g)| + (2^J\beta)^{\frac{n}{p}} \|g\|_*\\
			& \quad + \sum_{m=1}^{J} \| \chi_{B(0,2^{k+1}\alpha) \backslash B(0,2^k\alpha)} |m_{2^mB(0,\beta)}(g)-m_{2^{m-1}B(0,\beta)}(g)| \|_{L^p({\mathbb R}^n)}.\\
	\end{align*}
	Thus from (\ref{123}), we obtain 
	\begin{align*}
		&\| \chi_{B(0,2^{k+1}\alpha) \backslash B(0,2^k\alpha)} |g-g(y)| \|_{{\mathcal M}{}^p_q({\mathbb R}^n)}\\
		&\lesssim (2^{k+1}\alpha)^{\frac{n}{p}} |g(y)-m_{B(0,\beta)}(g)| + (2^J\beta)^{\frac{n}{p}} \|g\|_* + \sum_{m=1}^{J}(2^k\alpha)^{\frac{n}{p}} \|g\|_* \\
		&\sim (2^k\alpha)^{\frac{n}{p}} \left\{ \|g\|_* \cdot J + |g(y) - m_{B(0,\beta)}(g)| \right\}\\
		&\sim (2^k\alpha)^{\frac{n}{p}} \left\{ \|g\|_* \log_2 \left(\frac{2^{k+1}\alpha}{\beta} \right) + |g(y) - m_{B(0,\beta)}(g)| \right\},
	\end{align*}
	as required.
\end{proof}

Also, combining Propositions in Section 2, we immediately acquire the following Lemma; 

\begin{lemma}\label{apr}
Let $l \in \mathbb N$, $j=1,...,l$, $a_j \in {\rm BMO}({\mathbb R}^n)$, $1 < q \le p < \infty$, $0 < \alpha < n$, $\frac{1}{s}=\frac{1}{p}-\frac{\alpha}{n}$ and $\frac{p}{q}=\frac{s}{t}$. If
	\begin{align*}
		\forall a_1 \in C^\infty_{\rm c}({\mathbb R}^n), \quad [a_{\{1, ... , l\}},I_{\alpha}]: {\mathcal M}^{p}_q({\mathbb R}^n) \rightarrow \widetilde{\mathcal M}^{s}_t({\mathbb R}^n):bounded
	\end{align*}
then, 
	\begin{align*}
		\forall a_1 \in {\rm VMO}({\mathbb R}^n), \quad [a_{\{1, ... , l\}},I_{\alpha}]: {\mathcal M}^{p}_q({\mathbb R}^n) \rightarrow \widetilde{\mathcal M}^{s}_t({\mathbb R}^n):compact.
	\end{align*}

\end{lemma}

\begin{proof}
Let $a_1 \in {\rm VMO}({\mathbb R}^n)$. Then we can find a sequence $\{ a_{1_m}\}_{m=1}^{\infty} \subset C^\infty_{\rm c}({\mathbb R}^n)$ such that $\|a_1-a_{1_m} \|_* \leq m^{-1}$ holds. The compactness of $[a_{\{1_l,2, ... , l\}},I_{\alpha}]$ immediately follows from the natural embedding $\widetilde{\mathcal M}^{s}_t({\mathbb R}^n) \hookrightarrow {\mathcal M}^{s}_t({\mathbb R}^n)$. By Proposition \ref{mlc-bdd}, we deduce 
	\begin{align*}
		&\| [a_{\{1, ... , l\}},I_{\alpha}] - [a_{\{1_m,2, ... , l\}},I_{\alpha}] \|_{{\mathcal M}^p_q({\mathbb R}^n) \rightarrow {\mathcal M}^s_t({\mathbb R}^n)} \\
		&= \| [a_1-a_{1_m}, [a_{\{2, ... , l\}},I_{\alpha}]] \|_{{\mathcal M}^p_q({\mathbb R}^n) \rightarrow {\mathcal M}^s_t({\mathbb R}^n)} \\
		&\leq C\|a_1-a_{1_m} \|_* \prod_{j=2}^l \|a_j\|_*\\
		&\le C m^{-1} \prod_{j=2}^l \|a_j\|_*,
	\end{align*}
where $C$ is independent of $a_1$. \\
Thus, once we establish the boundedness of $[a_{\{1_m,2, ... , l\}},I_{\alpha}]$, then the compactness of $[a_{\{1, ... , l\}},I_{\alpha}]$  immediately follows; the norm-limit of compact operators is compact. 

\end{proof}

\section{Proof of Theorem \ref{main}}
If $l=1$, then the desired result follows from Proposition \ref{l1}. Thus, suppose  $l \ge 2$.\\
By Proposition \ref{decom} and Lemma \ref{apr}, we should check $[a_{\{1, ... , l\}},I_{\alpha}]f \in \overset{*}{{\mathcal M}}{}^s_t({\mathbb R}^n)$ and $[a_{\{1, ... , l\}},I_{\alpha}]f \in \overline{{\mathcal M}}{}^s_t({\mathbb R}^n)$ for any $f \in {\mathcal M}^{p}_q({\mathbb R}^n)$ and $a_1 \in C^\infty_{\rm c}({\mathbb R}^n)$.\\
Take $\beta>0$ such that ${\rm supp}(a_1) \subset B(0,\beta)=:K$. Also for $r>0$, $rK=rB(0,\beta)$ means $B(0,r\beta)$.  Fix $\epsilon \in (0,1)$ and define parameters $a,b$ as \\
\[
	\frac{1}{a}=\frac{1}{p}-\frac{\alpha+\epsilon}{n}>0, \quad and \quad \frac{p}{q}=\frac{a}{b}.
\]
They satisfy $a,b > 1$ and $a > s$. 

\subsection{Proof of $[a_{\{1, ... , l\}},I_{\alpha}]f \in \overset{*}{{\mathcal M}}{}^s_t({\mathbb R}^n)$}
For $R\gg1$ and $(c,d) \in \{(s,t),(a,b)\}$, by Minkowski's inequality, we have
\begin{align*}
	&\| \chi_{{B(0,R)}^c} [a_{\{1, ... , l\}},I_{\alpha}]f \|_{{\mathcal M}^c_d({\mathbb R}^n)}\\
	&=\left\| \chi_{{B(0,R)}^c} \int_{{\mathbb R}^n} \prod_{j=2}^l \left(a_j(\cdot)-a_j(y)\right) \frac{a_1(y)f(y)}{|\cdot-y|^{n-\alpha}} dy\right\|_{{\mathcal M}^c_d({\mathbb R}^n)}\\
	&\le \int_{{\mathbb R}^n} \left\| \chi_{{B(0,R)}^c}\frac{\prod_{j=2}^l \left(a_j(\cdot)-a_j(y)\right)}{|\cdot-y|^{n-\alpha}}\right\|_{{\mathcal M}^c_d({\mathbb R}^n)} |a_1(y)f(y)|dy\\
	&\le \sum_{k=1}^{\infty}\int_{K} \left\| \chi_{\{2^{k-1}R<|\cdot|<2^kR\}}\frac{\prod_{j=2}^l \left(a_j(\cdot)-a_j(y)\right)}{|\cdot-y|^{n-\alpha}}\right\|_{{\mathcal M}^c_d({\mathbb R}^n)} |a_1(y)f(y)|dy\\
	&\lesssim_{a_1}\sum_{k=1}^{\infty} \frac{1}{(2^kR)^{n-\alpha}} \int_K \left\| \chi_{\{2^{k-1}R<|\cdot|<2^kR\}} \prod_{j=2}^l{|a_j(\cdot)-a_j(y)|}\right\|_{{\mathcal M}^c_d({\mathbb R}^n)} |f(y)|dy.\\
\end{align*}
Thus for $\frac{1}{c}=\frac{l-1}{u}$, from Lemma \ref{BMO dyadic} and H\"{o}lder's inequality, we obtain 
\begin{align*}
	&\| \chi_{{B(0,R)}^c} [a_{\{1, ... , l\}},I_{\alpha}]f \|_{{\mathcal M}^c_d({\mathbb R}^n)}\\
	&\lesssim \sum_{k=1}^{\infty} \frac{1}{(2^kR)^{n-\alpha}} \int_K \prod_{j=2}^{l} \left\| \chi_{{\{2^{k-1}R<|\cdot|<2^kR\}}} |a_j-a_j(y)| \right\|_{L^u({\mathbb R}^n)} |f(y)|dy\\
	&\lesssim \sum_{k=1}^{\infty} \frac{1}{(2^kR)^{n-\alpha}} \int_K (2^kR)^{\frac{n(l-1)}{u}}\prod_{j=2}^{l}\left\{\|a_j\|_* \log_2\left(\frac{2^kR}{\beta}\right)+|a_j(y)-m_K(a_j)|\right\} |f(y)|dy\\
    	&\le \sum_{k=1}^{\infty} \frac{(2^kR)^{\frac{n}{c}}}{(2^kR)^{n-\alpha}} \int_K \left\{\log_2\left(\frac{2^kR}{\beta}\right)\right\}^{l-1} \prod_{j=2}^{l}\Big\{\|a_j\|_* +|a_j(y)-m_K(a_j)|\Big\} |f(y)|dy\\
	&\le \sum_{k=1}^{\infty} \frac{(2^kR)^{\frac{n}{c}}}{(2^kR)^{n-\alpha}} \left\{\log_2\left(\frac{2^kR}{\beta}\right)\right\}^{l-1} \|f\|_{L^q(K)} \left\|\prod_{j=2}^{l}\{\|a_j\|_* +|a_j(y)-m_K(a_j)|\}\right\|_{L^{q'}(K)}\\
	&\lesssim_{\{a_j\},f} \sum_{k=1}^{\infty} (2^kR)^{\frac{n}{c}-n+\alpha} \left\{\log_2\left(\frac{2^kR}{\beta}\right)\right\}^{l-1},
\end{align*}
where $q'$ satisfies $1=\frac{1}{q}+\frac{1}{q'}$. 
If we take $\rho>0$ which satisfies
\[
	-\frac{n}{c}+n-\alpha > \rho > 0,
\]
then we have
\begin{align*}
	&\| \chi_{{B(0,R)}^c} [a_{\{1, ... , l\}},I_{\alpha}]f \|_{{\mathcal M}^c_d({\mathbb R}^n)}\\
	&\lesssim \sum_{k=1}^{\infty} (2^kR)^{\frac{n}{c}-n+\alpha} \left\{\left(\frac{2^kR}{\beta}\right)^{\frac{\rho}{l-1}}\right\}^{l-1}\\
	&\sim R^{\frac{n}{c}-n+\alpha+\rho}\\
	&\rightarrow 0 \quad (R \rightarrow \infty).
\end{align*}
By Proposition \ref{char}, for any $f \in {\mathcal M}^{p}_q({\mathbb R}^n)$ we get
\begin{align*}
	[a_{\{1, ... , l\}},I_{\alpha}]f \in \overset{*}{\mathcal M}{}^s_t({\mathbb R}^n).
\end{align*}
Also, we get for any $f \in {\mathcal M}^{p}_q({\mathbb R}^n)$
\begin{align}\label{A}
	\exists L > \beta \quad s.t. \quad \chi_{B(0,L)^c}[a_{\{1, ... , l\}},I_{\alpha}]f \in {\mathcal M}{}^a_b({\mathbb R}^n).
\end{align}

\subsection{Proof of $[a_{\{1, ... , l\}},I_{\alpha}]f \in \overline{{\mathcal M}}{}^s_t({\mathbb R}^n)$}
Secondly, we will show $[a_{\{1, ... , l\}},I_{\alpha}]f \in \overline{{\mathcal M}}{}^s_t({\mathbb R}^n)$. Decompose the operator $[a_{\{1, ... , l\}},I_{\alpha}]$ with $A$, $A_+$ and $A_-$ as follows:
\begin{align*}
	[a_{\{1, ... , l\}},I_{\alpha}]f=Af+A_+f+A_-f,
\end{align*}
where
\begin{equation*}
	Af(x):=\chi_{B(0,L)^c}(x) [a_{\{1, ... , l\}},I_{\alpha}]f(x),
\end{equation*}
\begin{equation*}
	A_+f(x):=\chi_{2K}(x) [a_{\{1, ... , l\}},I_{\alpha}]f(x),
\end{equation*}
\begin{equation*}	
	A_-f(x):=\chi_{(2K)^c}(x) \chi_{B(0,L)}(x) [a_{\{1, ... , l\}},I_{\alpha}]f(x).
\end{equation*}
From (\ref{A}), we immediately obtain $Af \in \overline{{\mathcal M}}{}^s_t({\mathbb R}^n)$.

To use Proposition \ref{barsubsp}, we will show $A_-f \in {\mathcal M}{}^a_b({\mathbb R}^n)$ and $A_+f \in {\mathcal M}{}^{a'}_{b'}({\mathbb R}^n)$, where $a',b'$ are suitable parameters.

Set the parameters $d', \delta, \theta, \gamma$ as $\frac{1}{a}=\frac{l-1}{d'}$, $\delta= {\rm max}(\frac{q}{1+q}n-\alpha,0)$, $\theta=\frac{q+1}{2} > 1$ and $\gamma=n-\theta(n-\alpha-\delta)$. Then we have $|f|^{\theta} \in {\mathcal M}^{\frac{p}{\theta}}_{\frac{q}{\theta}}({\mathbb R}^n)$. Find $\xi \in \mathbb R^n$ such that $|\xi| \ll 1$ and $I_{\gamma}(|f|^{\theta})(\xi) < \infty$. 
Recalling Lemma \ref{BMO dyadic} and Minkowski's inequality, we get 
\begin{align*}
	&\|A_-f\|_{{\mathcal M}^{a}_{b}({\mathbb R}^n)}\\
	&= \left\| \chi_{\{2\beta \le |\cdot| \le L\}} \int_K \frac{\prod_{j=1}^l (a_j-a_j(y))}{|\cdot - y|^{n-\alpha}}f(y)dy \right\|_{{\mathcal M}^{a}_{b}({\mathbb R}^n)}\\	
	&\lesssim_{a_1} \int_{K}\left\| \chi_{\{2\beta \le |\cdot| \le L\}} \frac{\prod_{j=2}^{l}|a_j-a_j(y)|}{|\cdot-y|^{n-\alpha}} \right\|_{{\mathcal M}^{a}_{b}({\mathbb R}^n)} |f(y)|dy\\
	&\le \sum_{k=1}^{[\log_2{\frac{L}{\beta}}]+1}\int_{K}\left\| \chi_{2^{k+1}K \backslash 2^kK} \frac{\prod_{j=2}^{l}|a_j-a_j(y)|}{|\cdot-y|^{n-\alpha}} \right\|_{{\mathcal M}^{a}_{b}({\mathbb R}^n)} |f(y)|dy.
\end{align*}
Since $|\cdot - y|^{\delta} \ge {\beta}^{\delta}$ and $|\cdot-y| \sim |\xi-y|$ for $1 \le k \le [\log_2{\frac{L}{\beta}}]+1$, we obtain
\begin{align*}
	&\|A_-f\|_{{\mathcal M}^{a}_{b}({\mathbb R}^n)}\\
	&\lesssim \sum_{k=1}^{[\log_2{\frac{L}{\beta}}]+1}\int_{K}\left\| \chi_{2^{k+1}K \backslash 2^kK} \frac{\prod_{j=2}^{l}|a_j-a_j(y)|}{|\cdot-y|^{n-\alpha-\delta}} \right\|_{{\mathcal M}^{a}_{b}({\mathbb R}^n)} |f(y)|dy\\
	&\lesssim \sum_{k=1}^{[\log_2{\frac{L}{\beta}}]+1}\int_{K}\prod_{j=2}^{l} \left\| \chi_{2^{k+1}K \backslash 2^kK} |a_j-a_j(y)| \right\|_{L^{d'}{(\mathbb R}^n)} \frac{|f(y)|}{|\xi-y|^{n-\alpha-\delta}}dy\\
\end{align*}
Due to Lemma \ref{BMO dyadic} and H\"{older's inequality}, we obtain
\begin{align*}
	&\|A_-f\|_{{\mathcal M}^{a}_{b}({\mathbb R}^n)}\\
	&\lesssim \sum_{k=1}^{[\log_2{\frac{L}{\beta}}]+1} \int_{K} \prod_{j=2}^{l} \Big\{2^{\frac{(k+1)n}{d'}}\|a_j\|_*+|a_j(y)-m_{K}(a_j)| \Big\}\frac{|f(y)|}{|\xi-y|^{n-\alpha-\delta}}dy\\
	&\lesssim_{\{a_j\}} \sum_{k=1}^{[\log_2{\frac{L}{\beta}}]+1} \sum_{B \subset \{2,...,l\}} \int_{K} \prod_{j \in B} |a_j(y)-m_{K}(a_j)| \frac{|f(y)|}{|\xi-y|^{n-\alpha-\delta}}dy\\
	&\lesssim \sum_{B \subset \{2,...,l\}} \left\| \prod_{j \in B}  |a_j-m_{K}(a_j)| \right\|_{L^{\frac{\theta}{\theta-1}}(K)} \times \left(\int_{K} \frac{|f(y)|^{\theta}}{|\xi-y|^{\theta(n-\alpha-\delta)}} dy \right)^{\frac{1}{\theta}}\\
	&\lesssim_{\{a_j\}} \left( I_{\gamma}(|f|^{\theta})(\xi) \right)^{\frac{1}{\theta}}\\
	&< \infty.
\end{align*}
By Proposition \ref{barsubsp}, we get 
\begin{align*}
	A_-f \in \overline{{\mathcal M}}{}^s_t({\mathbb R}^n).
\end{align*}

To prove $A_+f \in {\mathcal M}{}^{a'}_{b'}({\mathbb R}^n)$, we set
\begin{align*}
	A_+f(x)&={\rm I}+{\rm II},
\end{align*}
where
\begin{align*}
	{\rm I}:&=\chi_{2K}(x)\int_{y \in 2K} \prod_{j=1}^l(a_j(x)-a_j(y))\frac{f(y)}{|x-y|^{n-\alpha}}dy,\\
	{\rm II}:&=\chi_{2K}(x)\int_{y \notin 2K} \prod_{j=1}^l(a_j(x)-a_j(y))\frac{f(y)}{|x-y|^{n-\alpha}}dy.
\end{align*}
Since $a_1 \in C^\infty_{\rm c}({\mathbb R}^n)$, by the mean value theorem and the fact that $a \in L^{\infty}({\mathbb R}^n)$ we get for any $0<\lambda<1$,
\[
	|a_1(x)-a_1(y)| \lesssim |x-y|^{\lambda}.
\]
Then, 
\begin{align*}
	|{\rm I}|&\lesssim \chi_{2K}(x)\int_{y \in 2K} \prod_{j=2}^l|a_j(x)-a_j(y)|\frac{|f(y)|}{|x-y|^{n-\alpha-\frac{\epsilon}{2}}}dy\\
	&\le \chi_{2K}(x)\int_{y \in 2K} \prod_{j=2}^l \left\{|a_j(x)|+|a_j(y)|\right\}\frac{|f(y)|}{|x-y|^{n-\alpha-\frac{\epsilon}{2}}}dy\\
	&= \chi_{2K}(x) \sum_{B \subset \{2,...,l\}}\prod_{j \in B}|a_j(x)| \int_{y \in 2K} \prod_{j \notin B}|a_j(y)|\frac{|f(y)|}{|x-y|^{n-\alpha-\frac{\epsilon}{2}}}dy.
\end{align*}
As a model case, we will consider
\[
	\chi_{2K}(x) \prod_{j=2}^m|a_j(x)| \int_{y \in 2K} \prod_{j=m+1}^l|a_j(y)|\frac{|f(y)|}{|x-y|^{n-\alpha-\frac{\epsilon}{2}}}dy
\]
 for any $1 \le m \le l$. Here, it will be understood that $\prod_{j=2}^{1}|a_j(x)|=\prod_{j=l+1}^{l}|a_j(x)|=1$. 
Define $a',a'',b'$ by 
\begin{align*}
	\frac{1}{a'}=\frac{1}{p}-\frac{\alpha+\frac{\epsilon}{3}}{n},\quad \frac{1}{a''}=\frac{1}{p}-\frac{\alpha+\frac{\epsilon}{2}}{n}, \quad {\rm and} \quad  \frac{p}{q}=\frac{a'}{b'}.
\end{align*}
Then, $a', a'', a$ satisfy
\[
	a' < a'' < a.
\]
Now, take $\tilde{a} \in (a',a'')$ with $0 < \frac{1}{\tilde{a}}+\frac{\alpha+\frac{\epsilon}{2}}{n} < 1$ and define $\tilde{p}$, $\tilde{q}$, $\tilde{b}$, and $r, w$ by
\begin{align*}
	\frac{1}{\tilde{a}}=\frac{1}{\tilde{p}}-\frac{\alpha+\frac{\epsilon}{2}}{n}, \quad& \frac{p}{q}=\frac{\tilde{a}}{\tilde{b}}=\frac{\tilde{p}}{\tilde{q}}, \\ 
	\frac{1}{a'}=\frac{m-1}{r}+\frac{1}{\tilde{a}}, \quad {\rm and}& \quad \frac{1}{\tilde{p}}=\frac{l-m}{w}+\frac{1}{p}.
\end{align*}
Then, we have $\tilde{p} < p$. Thus we get
\begin{align*}
    	&\left\| \chi_{2K}(\cdot) \prod_{j=2}^m|a_j(\cdot)| \int_{y \in 2K} \prod_{j=m+1}^l|a_j(y)|\frac{|f(y)|}{|\cdot-y|^{n-\alpha-\frac{\epsilon}{2}}}dy\right\|_{{\mathcal M}^{a'}_{b'}({\mathbb R}^n)}\\
	&\le \prod_{j=2}^{m} \left\| \chi_{2K}(\cdot) a_j\right\|_{L^r({\mathbb R}^n)} \left\| \chi_{2K}(\cdot) \int_{y \in {\mathbb R}^n} \chi_{2K}(y) \prod_{j=m+1}^l |a_j(y)|\frac{|f(y)|}{|\cdot-y|^{n-\alpha-\frac{\epsilon}{2}}}dy\right\|_{{\mathcal M}^{\tilde{a}}_{\tilde{b}}({\mathbb R}^n)}\\
	&\lesssim_{\{a_j\}} \left\| I_{\alpha+\frac{\epsilon}{2}}\left(\chi_{2K}(\prod_{j=m+1}^{l}|a_j|)|f|\right) \right\|_{{\mathcal M}^{\tilde{a}}_{\tilde{b}}({\mathbb R}^n)}\\
	&\lesssim \left\| \chi_{2K}\left(\prod_{j=m+1}^{l}|a_j|\right)|f| \right\|_{{\mathcal M}^{\tilde{p}}_{\tilde{q}}({\mathbb R}^n)}\\
	&\le \prod_{j=m+1}^{l} \|a_j\|_{L^w(2K)} \|f\|_{{\mathcal M}^{p}_{q}({\mathbb R}^n)}\\
	&< \infty.
\end{align*}
Hence, we get
\[
	{\rm I} \in {\mathcal M}^{a'}_{b'}({\mathbb R}^n).
\]
On the other hand, using Minkowski's inequality, we have
\begin{align*}
	&\|{\rm II}\|_{{\mathcal M}^{a'}_{b'}({\mathbb R}^n)}\\
	&=\left\| \chi_{2K}(\cdot) a_1(\cdot) \int_{y \notin 2K} \prod_{j=2}^{l}(a_j(\cdot)-a_j(y)) \frac{f(y)}{|\cdot - y|^{n-\alpha}}dy \right\|_{{\mathcal M}^{a'}_{b'}({\mathbb R}^n)}\\
	&=\left\| \chi_K(\cdot)a_1(\cdot) \sum_{k=1}^{\infty}  \int_{2^{k+1}K \backslash 2^kK} \prod_{j=2}^{l}(a_j(\cdot)-a_j(y)) \frac{f(y)}{|\cdot - y|^{n-\alpha}}dy \right\|_{{\mathcal M}^{a'}_{b'}({\mathbb R}^n)}\\
	&\lesssim_{a_1} \sum_{k=1}^{\infty} \int_{2^{k+1}K \backslash 2^kK} \left\| \chi_K(\cdot) \frac{ \prod_{j=2}^{l}(a_j(\cdot)-a_j(y))}{|\cdot - y|^{n-\alpha}}\right\|_{{\mathcal M}^{a'}_{b'}({\mathbb R}^n)} |f(y)|dy.
\end{align*}
Thus, we concentrate on estimating the following;   
\[
	{\rm II}_k:=\left\| \chi_K(\cdot) \frac{ \prod_{j=2}^{l}(a_j(\cdot)-a_j(y))}{|\cdot - y|^{n-\alpha}}\right\|_{{\mathcal M}^{a'}_{b'}({\mathbb R}^n)}, y \in 2^{k+1}K \backslash 2^kK.
\]
Since we have $|\cdot - y| \sim 2^k\beta$, for $\frac{1}{a'}=\frac{l-1}{c'}$ we obtain
\begin{align*}
	{\rm II}_k&\sim \left\| \chi_K \prod_{j=2}^{l}|a_j(\cdot)-a_j(y)| \right\|_{{\mathcal M}^{a'}_{b'}({\mathbb R}^n)}\times \frac{1}{(2^k\beta)^{n-\alpha}}\\
	&\le \prod_{j=2}^{l} \left\| \chi_K |a_j-a_j(y)| \right\|_{L^{c'}({\mathbb R}^n)} \times \frac{1}{(2^k\beta)^{n-\alpha}}.
\end{align*}
Using the triangle inequality, we have 
\begin{align*}
	&\left\| a_j-a_j(y) \right\|_{L^{c'}(K)}\\
	&=\left\| a_j - m_K(a_j) + \sum_{m=0}^{k}(m_{2^mK}(a_j)-m_{2^{m+1}K}(a_j)) + m_{2^{k+1}K}(a_j)-a_j(y) \right\|_{L^{c'}(K)}\\
	&\le \left\|a_j-m_K(a_j)\right\|_{L^{c'}(K)} + \left\|m_{2^{k+1}K}(a_j)-a_j(y)\right\|_{L^{c'}(K)}\\
		&\quad + \sum_{m=0}^{k} \left\|m_{2^mK}(a_j)-m_{2^{m+1}K}(a_j) \right\|_{L^{c'}(K)}\\
	&\le |K|^{\frac{1}{c'}} \|a_j\|_* + |K|^{\frac{1}{c'}} |m_{2^{k+1}K}(a_j)-a_j(y)| + \sum_{m=0}^{k} 2^{n+1} |K|^{\frac{1}{c'}} \|a_j\|_*\\
	&\lesssim k\|a_j\|_* + |m_{2^{k+1}K}(a_j)-a_j(y)|.
\end{align*}
Combining these inequalities and using H\"{o}lder's inequality, we get
\begin{align*}
	&\|{\rm II}\|_{{\mathcal M}^{a'}_{b'}({\mathbb R}^n)}\\
	&\lesssim \sum_{k=1}^{\infty} \int_{2^{k+1}K \backslash 2^kK} {\rm II}_k |f(y)|dy\\
	&\lesssim \sum_{k=1}^{\infty} \int_{2^{k+1}K \backslash 2^kK} \prod_{j=2}^{l} \Big\{ k\|a_j\|_* + |m_{2^{k+1}K}(a_j)-a_j(y)| \Big\} \frac{|f(y)|}{(2^k\beta)^{n-\alpha}}dy\\
	&\lesssim \sum_{k=1}^{\infty} \frac{1}{2^{k(n-\alpha)}} \|f\|_{L^{\eta}(2^{k+1}K \backslash 2^kK)} \prod_{j=2}^{l} \Big\| k\|a_j\|_* + |m_{2^{k+1}K}(a_j)-a_j| \Big\|_{L^{\mu}(2^{k+1}K \backslash 2^kK)},
\end{align*}
where
\[
	1<\eta<q, \quad {\rm and} \quad 1=\frac{1}{\eta}+\frac{l-1}{\mu}.
\]
Also, by triangle inequality we get
\begin{align*}
	&\Big\| k\|a_j\|_* + |m_{2^{k+1}K}(a_j)-a_j| \Big\|_{L^{\mu}(2^{k+1}K \backslash 2^kK)}\\
	&\lesssim k\|a_j\|_* (2^k\beta)^{\frac{n}{\mu}} + (2^k\beta)^{\frac{n}{\mu}} \|a_j\|_* \\
	&\lesssim \|a_j\|_*2^{\frac{kn}{\mu}}k.
\end{align*}
Thus,
\begin{align*}
	&\|{\rm II}\|_{{\mathcal M}^{a'}_{b'}({\mathbb R}^n)}\\
	&\lesssim \sum_{k=1}^{\infty} \frac{1}{2^{k(n-\alpha)}} \|f\|_{L^{\eta}(2^{k+1}K \backslash 2^kK)} \prod_{j=2}^{l}\{ \|a_j\|_*k2^{\frac{kn}{\mu}} \}\\
	&\sim_{\{a_j\}} \sum_{k=1}^{\infty} \frac{1}{2^{k(n-\alpha)}} \|f\|_{L^{\eta}(2^{k+1}K \backslash 2^kK)} k^{l-1} 2^{\frac{(l-1)kn}{\mu}}\\
	&\lesssim \sum_{k=1}^{\infty} |2^{k+1}K|^{\frac{1}{p}-\frac{1}{\eta}} \|f\|_{L^{\eta}(2^{k+1}K)} \frac{1}{2^{k(n-\alpha)}} k^{l-1} 2^{\frac{(l-1)kn}{\mu}} (2^{kn})^{\frac{1}{\eta}-\frac{1}{p}}\\
	& \le \sum_{k=1}^{\infty} \|f\|_{{\mathcal M}^{p}_{\eta}({\mathbb R}^n)} k^{l-1} 2^{k(\alpha - \frac{n}{p})}\\
	&\lesssim \|f\|_{{\mathcal M}^{p}_{\eta}({\mathbb R}^n)}\\
	&\le \|f\|_{{\mathcal M}^{p}_{q}({\mathbb R}^n)}.
\end{align*}
Therefore, we get ${\rm II} \in {\mathcal M}{}^{a'}_{b'}({\mathbb R}^n)$. Hence, by Proposition \ref{barsubsp}, we now have $A_+f \in \overline{{\mathcal M}}{}^s_t({\mathbb R}^n)$. 

Since $A_+f, A_-f, Af \in \overline{{\mathcal M}}{}^s_t({\mathbb R}^n)$, we get $[a_{\{1, ... , l\}},I_{\alpha}]f \in \overline{{\mathcal M}}{}^s_t({\mathbb R}^n)$. 
Hence, we finally obtain
\[
	[a_{\{1, ... , l\}},I_{\alpha}]f \in \widetilde{\mathcal M}^{s}_t({\mathbb R}^n).
\] 

\end{document}